\documentclass[a4paper,12pt]{article}
\usepackage{graphicx}
\usepackage{float}
\usepackage{latexsym}
\usepackage{amssymb}
\usepackage{amsmath}
\usepackage{amsthm}
\usepackage{cite}
\usepackage{bm}
\usepackage{epstopdf}

\newtheorem{theorem}{Theorem}[section]

\newtheorem{example}[theorem]{Example}
\newtheorem{proposition}[theorem]{Proposition}
\newtheorem{remark}[theorem]{Remark}
\newtheorem{definition}[theorem]{Definition}

\begin{document}

\title{Non-parabolic Spatial Hybrid Framed Curves and Their Applications in the Spatial Hybrid Number Space}
\author{Kaixin Yao\\
{\small \it School of Science, Yanshan University, Qinhuangdao 066004, P. R. China}\\
{\small \it e-mail: yaokx@ysu.edu.cn}
  }

\date{\today}

\maketitle
\begin{abstract}
In this paper, we define non-parabolic spatial hybrid framed curves in the spatial hybrid number space, which may have singularities, and prove the existence and uniqueness theorem for non-parabolic spatial hybrid framed curves. As applications, we define evolutes, involutes, pedal and contrapedal curves of non-parabolic spatial hybrid framed curves and discuss their relations.
\end{abstract}

\renewcommand{\thefootnote}{\fnsymbol{footnote}}
\footnote[0]{2020 Mathematics Subject Classiffcation. 53A35, 57R45, 15A63.}

\section{Introduction}
The set of hybrid numbers, a non-communicative number system denoted by $\mathbb{H}$ or $\mathbb{K},$ unifies and generalizes complex numbers, dual numbers and hyperbolic numbers. The hybrid number system was defined by \"{O}zdemir in \cite{MO}. There have been a large quantity of results about algebra and geometry on the hybrid number system. In algebra, \"{O}zdemir researched representations of hybrid numbers such as the matrix representation, the polar representation and the exponential representation. He also gave formulas to find roots of the $n$-th degree of a hybrid number. Analogous to the similarity of matrices, \"{O}zt\"{u}rk and \"{O}zdemir defined the similarity of hybrid numbers and investigated properties of similar hybrid numbers \cite{IO}. The hybridian product can induce the scalar product $g$ of two hybrid numbers. The hybrid number system $\mathbb{H}$ endowed with the scalar product $g$ is called the hybrid number space. It is a four-dimensional linear space and isomorphic to the $(\mathbb{R}^4, g).$ All hybrid numbers whose scalar parts vanish form a subspace of $\mathbb{H},$ called the spatial hybrid number space. Akb\i y\i k gave Frenet-Serret formulas for spatial hybrid curves and hybrid curves \cite{MA}. These are basic and useful tools to deal with properties of regular curves.

In the history of differential geometry, people usually researched regular curves. But singular curves are common in our life, such as movement trajectories of objects and edges of buildings. We cannot establish the Frenet frame of a singular curve at its singluarity by traditional method since the derivarite of the curve is a zero vector here. Bishop proposed that there is more than one way to frame a curve \cite{RLB}. Although Bishop only considered regular curves, people found that his idea can be applied on singular curves. Fukunaga and Takahashi proved the existence and uniqueness theorems of Legendre curves, whose base curves may have singularities \cite{TF}. Increasing the dimension of the space, Honda and Takahashi defined framed curves in the Euclidean $n$-space \cite{SH1}. A framed curve is a smooth curve with $(n-1)$ orthogonal unit vectors satisfying the derivarite of the smooth curve is orthogonal to each unit vectors mentioned above at any point. One can find more appropriate frames for some framed curves. Honda and Takahashi arised the Frenet type frame for framed curves in the Euclidean 3-space \cite{SH2,SH3}. It can be regarded as the generalization of the Frenet frame for regular curves. So far questions related to framed curves can be handled and abundant results are produced \cite{BDY, OGY, KY1, KY2}.

In the present paper, we consider non-parabolic spatial hybrid framed curves and curves generated by them. In Section 2, we recall some basic knowledge of hybrid numbers. In Section 3, we define non-parabolic spatial hybrid framed curves in the spatial hybrid number space and give the fundamental theorem of non-parabolic spatial hybrid framed curves. As applications, we define curves generated by non-parabolic spatial hybrid framed curves and research their relations in Section 4. Finally, we give an example to show our results.

All maps and manifolds considered here are differentiable of class $C^\infty.$

\section{Preliminaries}
Let
$$\mathbb H = \{a + b \bm i + c \bm \varepsilon + d \bm h ~|~ a,b,c,d \in \mathbb R\}$$
be the set of hybrid numbers and
$$\mathbb H^p = \{b \bm i + c \bm \varepsilon + d \bm h \in \mathbb H \}$$
be a subset of $\mathbb H.$ $\mathbb H^p$ is a three-dimensional vector space on $\mathbb{R}$ with the basis $\{\bm i, \bm \varepsilon, \bm h\}.$ For any $\bm H_1 = a_1 + b_1 \bm i + c_1 \bm \varepsilon + d_1 \bm h,~ \bm H_2 = a_2 + b_2 \bm i +c_2 \bm \varepsilon + d_2 \bm h \in \mathbb H,$ their hybridian product is
$$\begin{aligned}
	\bm H_1\bm H_2 =& a_1a_2 - b_1b_2 + b_1c_2 + b_2c_1 + d_1d_2 + a_1(b_2 \bm i +c_2 \bm \varepsilon + d_2 \bm h) + a_2(b_1 \bm i +c_1 \bm \varepsilon + d_1 \bm h)\\
	&+ (b_1d_2 - b_2d_1)\bm i + (b_1d_2 - b_2d_1 - c_1d_2 + c_2d_1)\bm \varepsilon + (b_2c_1 - b_1c_2)\bm h.
\end{aligned}$$
The hybridian product is associative but non-commutative. If $\bm H_1,\bm H_2 \in \mathbb H^p,$ then their hybridian product can be simplified to
$$\begin{aligned}
	\bm H_1\bm H_2 =& -b_1b_2 + b_1c_2 + b_2c_1 + d_1d_2 + (b_1d_2 - b_2d_1)\bm i \\
	&+ (b_1d_2 - b_2d_1 - c_1d_2 + c_2d_1)\bm \varepsilon + (b_2c_1 - b_1c_2)\bm h.
\end{aligned}$$
The conjugate of $\bm H_1$ is $\bm {\bar H}_1 = a_1 - b_1 \bm i - c_1 \bm \varepsilon - d_1 \bm h.$
The scalar product of $\bm H_1$ and $\bm H_2$ is defined by
$$\begin{aligned}
	g(\bm H_1,\bm H_2) =& \frac{1}{2}(\bm H_1 \bm {\bar H}_2 + \bm H_2 \bm {\bar H}_1)\\
	=& a_1a_2 + b_1b_2 - b_1c_2 - b_2c_1 - d_1d_2.
\end{aligned}$$
When $\bm H_1,\bm H_2 \in \mathbb H^p,$
$$\begin{aligned}
	g(\bm H_1,\bm H_2) =& b_1b_2 - b_1c_2 - b_2c_1 - d_1d_2\\
	=& (b_1, c_1, d_1)
	\begin{pmatrix}
		1 & -1 & 0\\
		-1 & 0 & 0\\
		0 & 0 & -1
	\end{pmatrix}
	\begin{pmatrix}
		b_2\\
		c_2\\
		d_2	
	\end{pmatrix}.
\end{aligned}$$

We still denote $g|_{\mathbb H^p \times \mathbb H^p}$ by $g$ and call $(\mathbb H^p, g)$ the spatial hybrid number space. We can indentify $(\mathbb H^p, g)$ and $(\mathbb R^3, g).$ $\bm H_1$ and $\bm H_2$ are $g$-orthogonal if $g(\bm H_1,\bm H_2) = 0.$ For any $\bm H \in \mathbb H^p,$ it is called elliptic, hyperbolic or parabolic if $g(\bm H, \bm H)$ is positive, negative or zero, respectively. The norm of $\bm H$ is defined by $||\bm H|| = \sqrt{|g(\bm H, \bm H)|}.$ The vector product of $\bm H_1$ and $\bm H_2$ is defined by
$$\begin{aligned}
	\bm H_1 \times \bm H_2 =& \dfrac{1}{2}(\bm H_1 \bm {\bar H}_2 - \bm H_2 \bm {\bar H}_1)\\
	=&(b_2d_1 - b_1d_2)\bm i + (b_2d_1 - b_1d_2 - c_2d_1 + c_1d_2) \bm \varepsilon + (b_1c_2 - b_2c_1) \bm h\\
	=&(b_1, c_1, d_1)
	\begin{pmatrix}
		-d_2 & -d_2 & c_2\\
		0 & d_2 & -b_2\\
		b_2 & b_2-c_2 & 0
	\end{pmatrix}
	\begin{pmatrix}
		\bm{i}\\
		\bm{\varepsilon}\\
		\bm{h}
	\end{pmatrix}.
\end{aligned}$$

\begin{proposition}
	{\rm Take $\bm H_k = b_k \bm i + c_k \bm \varepsilon + d_k \bm h \in \mathbb H^p, ~ k = 1,2,3,4$, then
		\begin{itemize}
			\item [\rm (1)] $g(\bm H_1 \times \bm H_2,\bm H_3) = -\det
			\begin{pmatrix}
				b_1 & c_1 & d_1 \\
				b_2 & c_2 & d_2 \\
				b_3 & c_3 & d_3
			\end{pmatrix}.$
			\item [\rm (2)] $(\bm H_1 \times \bm H_2) \times \bm H_3 = g(\bm H_1, \bm H_3)\bm H_2 - g(\bm H_2, \bm H_3)\bm H_1.$
			\item [\rm (3)] $g(\bm H_1 \times \bm H_2, \bm H_3 \times \bm H_4) = \det
			\begin{pmatrix}
				g(\bm H_1, \bm H_3) & g(\bm H_1, \bm H_4) \\
				g(\bm H_2, \bm H_3) & g(\bm H_2, \bm H_4)
			\end{pmatrix}.$
		\end{itemize}
	}
\end{proposition}
\begin{proof}
	\begin{itemize}
		\item [\rm (1)] $$\begin{aligned}
			&g(\bm H_1 \times \bm H_2, \bm H_3)\\
			=& (b_2d_1 - b_1d_2, b_2d_1 - b_1d_2 - c_2d_1 + c_1d_2, b_1c_2 - b_2c_1)
			\begin{pmatrix}
				1 & -1 & 0\\
				-1 & 0 & 0\\
				0 & 0 & -1
			\end{pmatrix}
			\begin{pmatrix}
				b_3\\
				c_3\\
				d_3
			\end{pmatrix}\\
			=& (c_2d_1 - c_1d_2, b_1d_2 - b_2d_1, b_2c_1 - b_1c_2)
			\begin{pmatrix}
				b_3\\
				c_3\\
				d_3
			\end{pmatrix}\\
			=&-\det \begin{pmatrix}
				b_1 & c_1 & d_1 \\
				b_2 & c_2 & d_2 \\
				b_3 & c_3 & d_3
			\end{pmatrix}.
			\end{aligned}$$
		
		\item [\rm (2)] $$\begin{aligned}
			(\bm H_1 \times \bm H_2) \times \bm H_3
			=& (b_1, c_1, d_1)
			\begin{pmatrix}
				-d_2 & -d_2 & c_2\\
				0 & d_2 & -b_2\\
				b_2 & b_2-c_2 & 0
			\end{pmatrix}
			\begin{pmatrix}
				-d_3 & -d_3 & c_3\\
				0 & d_3 & -b_3\\
				b_3 & b_3-c_3 & 0
			\end{pmatrix}
			\begin{pmatrix}
				\bm{i}\\
				\bm{\varepsilon}\\
				\bm{h}
			\end{pmatrix}\\
			=& \begin{pmatrix}
				d_3(b_1d_2-b_2d_1) + b_3(b_1c_2 - b_2c_1)\\
				d_3(c_1d_2 - c_2d_1) + (b_3-c_3)(b_1c_2 - b_2c_1)\\
				(b_3-c_3)(b_1d_2 - b_2d_1) - b_3(c_1d_2 - c_2d_1)
			\end{pmatrix}^T
			\begin{pmatrix}
				\bm{i}\\
				\bm{\varepsilon}\\
				\bm{h}
			\end{pmatrix},\\
			g(\bm H_1, \bm H_3) \bm H_2 =& (b_1b_3 - b_1c_3 - b_3c_1 - d_1d_3)(b_2 \bm i + c_2 \bm \varepsilon + d_2 \bm h),\\
			g(\bm H_2, \bm H_3) \bm H_1 =& (b_2b_3 - b_2c_3 - b_3c_2 - d_2d_3)(b_1 \bm i + c_1 \bm \varepsilon + d_1 \bm h).
			\end{aligned}$$
			So
			$$(\bm H_1 \times \bm H_2) \times \bm H_3 = g(\bm H_1, \bm H_3)\bm H_2 - g(\bm H_2, \bm H_3)\bm H_1.$$
		
		\item [\rm (3)] $$\begin{aligned}
			&g(\bm H_1 \times \bm H_2, \bm H_3 \times \bm H_4)\\
			=& g((\bm H_3 \times \bm H_4) \times \bm H_1, \bm H_2)\\
			=& g\big(g(\bm H_3, \bm H_1)\bm H_4 - g(\bm H_4, \bm H_1)\bm H_3, \bm H_2\big)\\
			=& g(\bm H_1, \bm H_3)g(\bm H_2, \bm H_4) - g(\bm H_1, \bm H_4)g(\bm H_2, \bm H_3).
		\end{aligned}$$
	\end{itemize}
\end{proof}

\section{Non-parabolic spatial hybrid framed curves}
Denote $\Delta = \{(\bm\nu_1, \bm\nu_2) \in \mathbb H^p \times \mathbb H^p ~|~ ||\bm\nu_1|| = ||\bm\nu_2|| = 1, g(\bm\nu_1, \bm\nu_2) = 0\}.$ We define non-parabolic spatial hybrid framed curves in $\mathbb H^p.$

\begin{definition}
	{\rm Let $I$ be an interval in $\mathbb R.$ $(\gamma, \bm\nu_1, \bm\nu_2) : I \rightarrow \mathbb H^p \times \Delta$ is called a non-parabolic spatial hybrid framed curve if $g(\gamma'(t), \bm\nu_1(t)) = g(\gamma'(t), \bm\nu_2(t)) = 0$ for all $t \in I.$ We call $\gamma : I \rightarrow \mathbb H^p$ a non-parabolic spatial hybrid framed base curve if there exists $(\bm\nu_1, \bm\nu_2) : I \rightarrow \Delta$ such that $(\gamma, \bm\nu_1, \bm\nu_2)$ is a non-parabolic spatial hybrid framed curve.
	}
\end{definition}
Let $\bm\mu(t) = \bm\nu_1(t) \times \bm\nu_2(t).$ Then $\{\bm\nu_1(t), \bm\nu_2(t), \bm\mu(t)\}$ is a $g$-orthogonal frame along $\gamma.$ The Frenet type formulas are
$$\begin{pmatrix}
	\bm\nu_1'(t)\\
	\bm\nu_2'(t)\\
	\bm\mu'(t)
\end{pmatrix}
=
\begin{pmatrix}
	0 & l(t) & m(t) \\
	-\delta_1\delta_2l(t) & 0 & n(t) \\
	-\delta_2m(t) & -\delta_1n(t) & 0
\end{pmatrix}
\begin{pmatrix}
	\bm\nu_1(t)\\
	\bm\nu_2(t)\\
	\bm\mu(t)
\end{pmatrix},
~\gamma'(t) = \alpha(t) \bm\mu(t),$$
where
$$\delta_1 = g(\bm\nu_1(t), \bm\nu_1(t)), ~ \delta_2 = g(\bm\nu_2(t), \bm\nu_2(t)), ~ l(t) = \delta_2g(\bm\nu_1'(t), \bm\nu_2(t)),$$
$$m(t) = \delta_1\delta_2g(\bm\nu_1'(t), \bm\mu(t)), ~ n(t) = \delta_1\delta_2g(\bm\nu_2'(t), \bm\mu(t)), ~ \alpha(t) = \delta_1\delta_2g(\gamma'(t), \bm\mu(t)).$$
The map $(l, m, n, \alpha) : I \rightarrow \mathbb R^4$ is called the curvature of $(\gamma, \bm\nu_1, \bm\nu_2).$ We say $(\gamma, \bm\nu_1, \bm\nu_2)$ is an elliptic (or a hyperbolic) spatial hybrid framed curve if $\delta_1\delta_2 = 1 ~(\text{or} ~ \delta_1\delta_2 = -1).$

It is obvious that $t_0 \in I$ is a singularity of $\gamma$ if and only if $\alpha(t_0) = 0.$

Now we give the fundamental theory of non-parabolic spatial hybrid framed curves. We first define spatial hybrid motions.

\begin{definition}
	{\rm A matrix $A \in \mathbb R^{3 \times 3}$ is called a spatial hybrid motion if $A^TGA = G, \det(A) = 1,$ where
		$$G = \begin{pmatrix}
			1 & -1 & 0\\
			-1 & 0 & 0\\
			0 & 0 & -1
		\end{pmatrix}.$$
	}
\end{definition}

Spatial hybrid motions form a group and keep the scalar product and the vector product.

\begin{proposition}
	{\rm All spatial hybrid motions form a group.}
\end{proposition}
\begin{proof}
	Define $$\mathcal G = \{A \in \mathbb R^{3 \times 3} ~|~ A^TGA = G, \det(A) = 1\}.$$
	For any $A, B \in \mathcal G,$ we have
	$$(AB)^TG(AB) = B^TA^TGAB = B^TGB = G$$
	and
	$$\det(AB) = \det(A)\det(B) = 1.$$
	So $AB \in \mathcal G.$
	
	The associative law holds naturally. The unit matrix
	$\begin{pmatrix}
		1 & 0 & 0\\
		0 & 1 & 0\\
		0 & 0 & 1
	\end{pmatrix} \in \mathcal G$
	is the unit element.
	The inverse element of $A$ is $A^{-1}$ and
	$$(A^{-1})^TGA^{-1} = (A^T)^{-1}A^TGAA^{-1} = G.$$
	So $\mathcal G$ is a group.
\end{proof}

\begin{proposition}\label{prop}
	{\rm For any $A \in \mathcal G$ and any $\bm H_1, \bm H_2 \in \mathbb H^p,$ we have $g(A\bm H_1, A\bm H_2) = g(\bm H_1, \bm H_2)$ and $(A\bm H_1) \times (A\bm H_2) = A(\bm H_1 \times \bm H_2).$
	}
\end{proposition}
\begin{proof}
	$$g(A\bm H_1, A\bm H_2) = (A\bm H_1)^TG(A\bm H_2) = \bm H_1^TA^TGA\bm H_2 = \bm H_1^TG\bm H_2 = g(\bm H_1, \bm H_2).$$
	
	For any $\bm H_3 \in \mathbb H^p,$ we have
	$$\begin{aligned}
		&g((A\bm H_1) \times (A\bm H_2), \bm H_3)\\
		=& -\det(A\bm H_1, A\bm H_2, \bm H_3)\\
		=& -\det(G^{-1}A^TG)\det(A\bm H_1, A\bm H_2, \bm H_3)\\
		=& -\det(\bm H_1, \bm H_2, G^{-1}A^TG\bm H_3)\\
		=& g(\bm H_1 \times \bm H_2, G^{-1}A^TG\bm H_3)\\
		=& \bm H_3^TG^TA(G^{-1})^TG(\bm H_1 \times \bm H_2)\\
		=& \bm H_3^TGA(\bm H_1 \times \bm H_2)\\
		=& g(\bm H_3, A(\bm H_1 \times \bm H_2)).
	\end{aligned}$$
	This implies $(A\bm H_1) \times (A\bm H_2) = A(\bm H_1 \times \bm H_2).$
\end{proof}

\begin{definition}
	{\rm Let $(\gamma, \bm\nu_1, \bm\nu_2) : I \rightarrow \mathbb H^p \times \Delta$ and $(\tilde\gamma, \bm{\tilde\nu_1}, \bm{\tilde\nu_2}) : I \rightarrow \mathbb H^p \times \Delta$ be two non-parabolic spatial hybrid framed curves. We call them congruent through a spatial hybrid motion if there exists $A \in \mathcal G$ and $\bm H_0 \in \mathbb H^p$ such that
		$$\tilde\gamma(t) = A(\gamma(t)) + \bm H_0, ~ \bm{\tilde\nu}_1(t) = A(\bm\nu_1(t)), ~ \bm{\tilde\nu}_2(t) = A(\bm\nu_2(t))$$
		for any $t \in I.$
	}
\end{definition}

By Proposition \ref{prop}, we know when two non-parabolic spatial hybrid framed curves are congruent, their curvatures coincide. Next we claim the existence and uniqueness theorem of non-parabolic spatial hybrid framed curves.

\begin{theorem}[Existence and uniqueness theorem]
	{\rm Given smooth functions $l, m, n, \alpha : I \rightarrow \mathbb R$ and constants $\delta_1, \delta_2 \in \{1, -1\},$ there exists a non-parabolic spatial hybrid framed curve $(\gamma, \bm\nu_1, \bm\nu_2) : I \rightarrow \mathbb H^p \times \Delta,$ whose curvature is $(l, m, n, \alpha)$ and $g(\bm\nu_1(t), \bm\nu_1(t)) = \delta_1, ~ g(\bm\nu_2(t), \bm\nu_2(t)) = \delta_2.$ The curve is unique under spatial hybrid motions.
	}
\end{theorem}
\begin{proof}
	{\rm Fix some $t_0 \in I.$ Consider a system
		\begin{equation}\label{eq1}
			\begin{pmatrix}
				\bm\nu_1'(t)\\
				\bm\nu_2'(t)\\
				\bm\mu'(t)
			\end{pmatrix}
			=
			\begin{pmatrix}
				0 & l(t) & m(t) \\
				-\delta_1\delta_2l(t) & 0 & n(t) \\
				-\delta_2m(t) & -\delta_1n(t) & 0
			\end{pmatrix}
			\begin{pmatrix}
				\bm\nu_1(t)\\
				\bm\nu_2(t)\\
				\bm\mu(t)
			\end{pmatrix}
		\end{equation}
		with the initial value
		$$\begin{matrix}
			g(\bm\nu_1(t_0), \bm\nu_1(t_0)) = \delta_1, & g(\bm\nu_2(t_0), \bm\nu_2(t_0)) = \delta_2, & g(\bm\mu(t_0), \bm\mu(t_0)) = \delta_1\delta_2,\\
			g(\bm\nu_1(t_0), \bm\nu_2(t_0)) = 0, & g(\bm\nu_1(t_0), \bm\mu(t_0)) = 0, & g(\bm\nu_2(t_0), \bm\mu(t_0)) = 0.
		\end{matrix}$$
		$l, m, n$ are smooth, so the solution of equation (\ref{eq1}) exists.
		
		Let $(\bm\nu_1, \bm\nu_2, \bm\mu) : I \rightarrow \mathbb H^p \times \mathbb H^p \times \mathbb H^p$ be a solution of the equation (\ref{eq1}). Define functions $a_1, a_2, ..., a_6 : I \rightarrow \mathbb R$ by
		$$\begin{matrix}
			a_1(t) = g(\bm\nu_1(t), \bm\nu_1(t)), & a_2(t) = g(\bm\nu_2(t), \bm\nu_2(t)), & a_3(t) = g(\bm\mu(t), \bm\mu(t)),\\
			a_4(t) = g(\bm\nu_1(t), \bm\nu_2(t)), & a_5(t) = g(\bm\nu_1(t), \bm\mu(t)), & a_6(t) = g(\bm\nu_2(t), \bm\mu(t)).
		\end{matrix}$$
		Consider the following system
		$$\begin{pmatrix}
			a_1'(t)\\
			a_2'(t)\\
			a_3'(t)\\
			a_4'(t)\\
			a_5'(t)\\
			a_6'(t)
		\end{pmatrix}
		=
		\begin{pmatrix}
			0 & 0 & 0 & 2l(t) & 2m(t) & 0 \\
			0 & 0 & 0 & -2\delta_1\delta_2l(t) & 0 & 2n(t) \\
			0 & 0 & 0 & 0 & -2\delta_2m(t) & -2\delta_1n(t) \\
			-\delta_1\delta_2l(t) & l(t) & 0 & 0 & n(t) & m(t) \\
			-\delta_2m(t) & 0 & m(t) & -\delta_1n(t) & 0 & l(t) \\
			0 & -\delta_1n(t) & n(t) & -\delta_2m(t) & -\delta_1\delta_2l(t) & 0
		\end{pmatrix}
		\begin{pmatrix}
			a_1(t)\\
			a_2(t)\\
			a_3(t)\\
			a_4(t)\\
			a_5(t)\\
			a_6(t)
		\end{pmatrix}$$
		with the initial value
		$$a_1(t_0) = \delta_1, ~ a_2(t_0) = \delta_2, ~ a_3(t_0) = \delta_1\delta_2, ~ a_4(t_0) = 0, ~ a_5(t_0) = 0, ~ a_6(t_0) = 0.$$
		It has the unique solution
		$$a_1(t) = \delta_1, ~ a_2(t) = \delta_2, ~ a_3(t) = \delta_1\delta_2, ~ a_4(t) = 0, ~ a_5(t) = 0, ~ a_6(t) = 0.$$
		That is to say $(\bm\nu_1, \bm\nu_2, \bm\mu) : I \rightarrow \Delta \times \mathbb H^p$ and $g(\bm\nu_1(t), \bm\mu(t)) = g(\bm\nu_2(t), \bm\mu(t)) = 0.$
		
		Take $\gamma(t) = \int \alpha(t)\bm\mu(t) {\rm d} t,$ then $(\gamma, \bm\nu_1, \bm\nu_2) : I \rightarrow \mathbb H^p \times \Delta$ is a non-parabolic spatial hybrid framed curve with the curvature $(l, m, n, \alpha)$ and $g(\bm\nu_1(t), \bm\nu_1(t)) = \delta_1, ~ g(\bm\nu_2(t), \bm\nu_2(t)) = \delta_2.$
		
		Let $(\gamma, \bm\nu_1, \bm\nu_2) : I \rightarrow \mathbb H^p \times \Delta$ and $(\tilde\gamma, \bm{\tilde\nu}_1, \bm{\tilde\nu}_2) : I \rightarrow \mathbb H^p \times \Delta$ be two non-parabolic spatial hybrid framed curves with the same curvature $(l, m, n, \alpha)$ and
		$$g(\bm\nu_1(t), \bm\nu_1(t)) = \delta_1, ~ g(\bm\nu_2(t), \bm\nu_2(t)) = \delta_2, ~ g(\bm{\tilde\nu}_1(t), \bm{\tilde\nu}_1(t)) = \delta_1, ~ g(\bm{\tilde\nu}_2(t), \bm{\tilde\nu}_2(t)) = \delta_2.$$
		Then there exists $A \in \mathcal G$ and $\bm H_0 \in \mathbb H^p$ such that
		$$\tilde\gamma(t_0) = A(\gamma(t_0)) + \bm H_0, ~ \bm{\tilde\nu}_1(t_0) = A(\bm\nu_1(t_0)), ~ \bm{\tilde\nu}_2(t_0) = A(\bm\nu_2(t_0))$$
		for some $t_0 \in I.$ Then
		$$\bm{\tilde\mu}(t_0) = \bm{\tilde\nu}_1(t_0) \times \bm{\tilde\nu}_2(t_0) = A(\bm\nu_1(t_0) \times \bm\nu_2(t_0)) = A(\bm\mu(t_0)).$$
		
		$(A\gamma + \bm H_0, A\bm\nu_1, A\bm\nu_2, A\bm\mu) : I \rightarrow \mathbb H^p \times \Delta \times \mathbb H^p$ and $(\tilde\gamma, \bm{\tilde\nu}_1, \bm{\tilde\nu}_2, \bm{\tilde\mu}) : I \rightarrow \mathbb H^p \times \Delta \times \mathbb H^p$ are both solutions of the system
		$$\left\{
		\begin{array}{l}
			\gamma'(t) = \alpha(t)\bm\mu(t),\\
			\bm\nu_1'(t) = l(t)\bm\nu_2(t) + m(t)\bm\mu(t),\\
			\bm\nu_2'(t) = -\delta_1\delta_2l(t)\bm\nu_1(t) + n(t)\bm\mu(t),\\
			\bm\mu'(t) = -\delta_2m(t)\bm\nu_1(t) - \delta_1n(t)\bm\nu_2(t).
		\end{array}
		\right.$$
		So $(A\gamma + \bm H_0, A\bm\nu_1, A\bm\nu_2, A\bm\mu) = (\tilde\gamma, \bm{\tilde\nu}_1, \bm{\tilde\nu}_2, \bm{\tilde\mu}).$ That is $(\gamma, \bm\nu_1, \bm\nu_2)$ and $(\tilde\gamma, \bm{\tilde\nu}_1, \bm{\tilde\nu}_2)$ are congruent under spatial hybrid motions.
	}
\end{proof}

For a non-parabolic spatial hybrid framed curve $(\gamma, \bm\nu_1, \bm\nu_2) : I \rightarrow \mathbb H^p \times \Delta$ with the curvature $(l, m, n, \alpha),$ any linear combination of $\bm\nu_1(t)$ and $\bm\nu_2(t)$ is $g$-orthogonal to $\gamma'(t).$ When $\delta_1 m^2(t) + \delta_2 n^2(t) \neq 0$ for all $t \in I,$ we take
$$\begin{pmatrix}
	\bm n_1(t)\\
	\bm n_2(t)
\end{pmatrix}
=
\frac{1}{\sqrt{\sigma(\delta_1 m^2(t) + \delta_2 n^2(t))}}
\begin{pmatrix}
	\sigma\delta_1 m(t) & \sigma\delta_2 n(t) \\
	-n(t) & m(t)
\end{pmatrix}
\begin{pmatrix}
	\bm\nu_1(t)\\
	\bm\nu_2(t)
\end{pmatrix},$$
where $\sigma = {\rm sgn}(\delta_1 m^2(t) + \delta_2 n^2(t)).$
Then\\
$$\begin{aligned}
	&g(\bm n_1(t), \bm n_1(t))\\
	=& \frac{1}{\sigma(\delta_1 m^2(t) + \delta_2 n^2(t))}g(\sigma\delta_1 m(t)\bm\nu_1(t) + \sigma\delta_2 n(t)\bm\nu_2(t), \sigma\delta_1 m(t)\bm\nu_1(t) + \sigma\delta_2 n(t)\bm\nu_2(t))\\
	=& \frac{1}{\sigma(\delta_1 m^2(t) + \delta_2 n^2(t))}(\delta_1 m^2(t) + \delta_2 n^2(t))\\
	=& \sigma,\\
	&g(\bm n_2(t), \bm n_2(t))\\
	=& \frac{1}{\sigma(\delta_1 m^2(t) + \delta_2 n^2(t))}g(-n(t)\bm\nu_1(t) + m(t)\bm\nu_2(t), -n(t)\bm\nu_1(t) + m(t)\bm\nu_2(t))\\
	=& \frac{1}{\sigma(\delta_1 m^2(t) + \delta_2 n^2(t))}(\delta_1 n^2(t) + \delta_2 m^2(t))\\
	=& \sigma\delta_1\delta_2,\\
	&g(\bm n_1(t), \bm n_2(t))\\
	=& \frac{1}{\sigma(\delta_1 m^2(t) + \delta_2 n^2(t))}g(\sigma\delta_1 m(t)\bm\nu_1(t) + \sigma\delta_2 n(t)\bm\nu_2(t), -n(t)\bm\nu_1(t) + m(t)\bm\nu_2(t))\\
	=& 0,\\
	&\bm n_1(t) \times \bm n_2(t)\\
	=& \frac{1}{\sigma(\delta_1 m^2(t) + \delta_2 n^2(t))} (\sigma\delta_1 m(t)\bm\nu_1(t) + \sigma\delta_2 n(t)\bm\nu_2(t)) \times (-n(t)\bm\nu_1(t) + m(t)\bm\nu_2(t))\\
	=& \bm\mu(t).
\end{aligned}$$

So $\{\bm n_1(t), \bm n_2(t), \bm\mu(t)\}$ is also a $g$-orthogonal frame along $\gamma.$ The Frenet type formulas are
$$\begin{pmatrix}
	\bm n_1'(t)\\
	\bm n_2'(t)\\
	\bm\mu'(t)
\end{pmatrix}
=
\begin{pmatrix}
	0 & L(t) & M(t) \\
	-\delta_1\delta_2 L(t) & 0 & 0 \\
	-\sigma\delta_1\delta_2 M(t) & 0 & 0
\end{pmatrix}
\begin{pmatrix}
	\bm n_1(t)\\
	\bm n_2(t)\\
	\bm\mu(t)
\end{pmatrix},
~\gamma'(t) = \alpha(t) \bm\mu(t),$$
where
$$L(t) = \frac{\delta_1\delta_2m(t)n'(t) - \delta_1\delta_2m'(t)n(t)}{\sigma(\delta_1 m^2(t) + \delta_2 n^2(t))} + \sigma\delta_1 l(t), ~ M(t) = \sqrt{\sigma(\delta_1 m^2(t) + \delta_2 n^2(t))}.$$
The $\bm n_1$ direction and $\bm n_2$ direction are called the principal normal direction and the binormal direction of $\gamma.$

\section{Curves generated by non-parabolic spatial hybrid framed curves}

In this section, we define evolutes of non-parabolic spatial hybrid framed curves using distance squared functions. As its inverse process, we define involutes and discuss their relationship.

\begin{definition}\label{def}
	{\rm Let $(\gamma, \bm\nu_1, \bm\nu_2) : I \rightarrow \mathbb H^p \times \Delta$ be a non-parabolic spatial hybrid framed curve. The distance squared function of $(\gamma, \bm\nu_1, \bm\nu_2)$ is defined by $F : I \times \mathbb H^p \rightarrow \mathbb R,$
		$$F(t,\bm{x}) = g(\gamma(t) - \bm{x}, \gamma(t) - \bm{x}).$$
		For some fixed $\bm{x} \in \mathbb H^p,$ $F(t,\bm{x})$ is denoted by $f_x(t).$
	}
\end{definition}

\begin{proposition}
	{\rm Let $(\gamma, \bm\nu_1, \bm\nu_2) : I \rightarrow \mathbb H^p \times \Delta$ be a non-parabolic spatial hybrid framed curve with the curvature $(l, m, n, \alpha).$ Assume
		$\alpha(t)\left(m(t)n'(t) - m'(t)n(t) + \delta_1\delta_2 l(t)m^2(t) + l(t)n^2(t)\right) \neq 0$
		for any $t \in I.$ Then the followings hold.
		\begin{enumerate}
			\item [\rm (1)] $f_x'(t_0) = 0$ if and only if there exist $\lambda_1, \lambda_2 \in \mathbb R$ such that $\gamma(t_0) - \bm{x} = \lambda_1\bm\nu_1(t_0) + \lambda_2\bm\nu_2(t_0).$
			
			\item[\rm (2)] $f_x'(t_0) = f_x''(t_0)= 0$ if and only if $\gamma(t_0) - \bm{x} = \lambda_1\bm\nu_1(t_0) + \lambda_2\bm\nu_2(t_0)$ and $\lambda_1 m(t_0) + \lambda_2 n(t_0) = \alpha(t_0).$
			
			\item[\rm (3)] $f_x'(t_0) = f_x''(t_0) = f_x'''(t_0) = 0$ if and only if
			$$\begin{aligned}
				\gamma(t_0) - \bm{x} = ~& \frac{n'(t_0)\alpha(t_0) - n(t_0)\alpha'(t_0) + \delta_1\delta_2 l(t_0)m(t_0)\alpha(t_0)}{m(t_0)n'(t_0) - m'(t_0)n(t_0) + \delta_1\delta_2 l(t_0)m^2(t_0) + l(t_0)n^2(t_0)}\bm\nu_1(t_0)\\
				+& \frac{m(t_0)\alpha'(t_0) - m'(t_0)\alpha(t_0) + l(t_0)n(t_0)\alpha(t_0)}{m(t_0)n'(t_0) - m'(t_0)n(t_0) + \delta_1\delta_2 l(t_0)m^2(t_0) + l(t_0)n^2(t_0)}\bm\nu_2(t_0).
			\end{aligned}$$
		\end{enumerate}
	}
\end{proposition}
\begin{proof}
	By Definition \ref{def}, we can get
	$$\begin{aligned}
		\frac{1}{2}f_x'(t) =& \alpha(t)g(\gamma(t) - \bm{x}, \bm\mu(t)),\\
		\frac{1}{2}f_x''(t) =& \alpha'(t)g(\gamma(t) - \bm{x}, \bm\mu(t)) + \delta_1\delta_2\alpha^2(t) - \alpha(t)g(\gamma(t) - \bm{x}, \delta_2m(t)\bm\nu_1(t) + \delta_1n(t)\bm\nu_2(t)),\\
		\frac{1}{2}f_x'''(t) =& \left(-\delta_2 m'(t)\alpha(t) + \delta_2 l(t)n(t)\alpha(t) - 2\delta_2m(t)\alpha'(t)\right)g(\gamma(t) - \bm{x}, \bm\nu_1(t))\\
		&- \left(\delta_1 n'(t)\alpha(t) + \delta_2 l(t)m(t)\alpha(t) + 2\delta_1n(t)\alpha'(t)\right)g(\gamma(t) - \bm{x}, \bm\nu_2(t))\\
		&- \left(\delta_2 m^2(t)\alpha(t) + \delta_1 n^2(t)\alpha(t) + \alpha''(t)\right)g(\gamma(t) - \bm{x}, \bm\mu(t)) + 3\delta_1\delta_2\alpha(t)\alpha'(t).
	\end{aligned}$$
	So
	\begin{enumerate}
		\item [\rm (1)] $f_x'(t_0) = 0$ if and only if there exist $\lambda_1, \lambda_2 \in \mathbb R$ such that $\gamma(t_0) - \bm{x} = \lambda_1\bm\nu_1(t_0) + \lambda_2\bm\nu_2(t_0).$
		
		\item[\rm (2)] $f_x'(t_0) = f_x''(t_0)= 0$ if and only if $\gamma(t_0) - \bm{x} = \lambda_1\bm\nu_1(t_0) + \lambda_2\bm\nu_2(t_0)$ and $\lambda_1 m(t_0) + \lambda_2 n(t_0) = \alpha(t_0).$
		
		\item[\rm (3)] $f_x'(t_0) = f_x''(t_0) = f_x'''(t_0) = 0$ if and only if $\gamma(t_0) - \bm{x} = \lambda_1\bm\nu_1(t_0) + \lambda_2\bm\nu_2(t_0)$ and
		$$\left\{
		\begin{array}{l}
			\lambda_1 m(t_0) + \lambda_2 n(t_0) = \alpha(t_0),\\
			\lambda_1\left(m'(t_0) - l(t_0)n(t_0)\right) + \lambda_2\left(n'(t_0) + \delta_1\delta_2 l(t_0)m(t_0)\right) = \alpha'(t_0).
		\end{array}
		\right.$$
		This means
		$$\begin{aligned}
			\gamma(t_0) - \bm{x} = ~& \frac{n'(t_0)\alpha(t_0) - n(t_0)\alpha'(t_0) + \delta_1\delta_2 l(t_0)m(t_0)\alpha(t_0)}{m(t_0)n'(t_0) - m'(t_0)n(t_0) + \delta_1\delta_2 l(t_0)m^2(t_0) + l(t_0)n^2(t_0)}\bm\nu_1(t_0)\\
			+& \frac{m(t_0)\alpha'(t_0) - m'(t_0)\alpha(t_0) + l(t_0)n(t_0)\alpha(t_0)}{m(t_0)n'(t_0) - m'(t_0)n(t_0) + \delta_1\delta_2 l(t_0)m^2(t_0) + l(t_0)n^2(t_0)}\bm\nu_2(t_0).
		\end{aligned}$$
	\end{enumerate}
\end{proof}

\begin{definition}
	{\rm Let $(\gamma, \bm\nu_1, \bm\nu_2) : I \rightarrow \mathbb H^p \times \Delta$ be a non-parabolic spatial hybrid framed curve with the curvature $(l, m, n, \alpha).$ Assume
		$m(t)n'(t) - m'(t)n(t) + \delta_1\delta_2 l(t)m^2(t) + l(t)n^2(t) \neq 0$
		for any $t \in I.$ The evolute of $(\gamma, \bm\nu_1, \bm\nu_2)$ is defined by $\mathcal Ev_\gamma : I \rightarrow \mathbb H^p,$
		$$\begin{aligned}
			\mathcal Ev_\gamma(t) = \gamma(t) &- \frac{n'(t)\alpha(t) - n(t)\alpha'(t) + \delta_1\delta_2 l(t)m(t)\alpha(t)}{m(t)n'(t) - m'(t)n(t) + \delta_1\delta_2 l(t)m^2(t) + l(t)n^2(t)}\bm\nu_1(t_0)\\
			&- \frac{m(t)\alpha'(t) - m'(t)\alpha(t) + l(t)n(t)\alpha(t)}{m(t)n'(t) - m'(t)n(t) + \delta_1\delta_2 l(t)m^2(t) + l(t)n^2(t)}\bm\nu_2(t_0).
		\end{aligned}$$
	}
\end{definition}

\begin{remark}
	{\rm  Let $(\gamma, \bm n_1, \bm n_2) : I \rightarrow \mathbb H^p \times \Delta$ be a non-parabolic spatial hybrid framed curve with the curvature $(L, M, 0, \alpha).$ Assume $L(t) \neq 0$ for any $t \in I.$ Then its evolute is
		$$\mathcal Ev_\gamma(t) = \gamma(t) - \frac{\alpha(t)}{M(t)}\bm n_1(t) - \delta_1\delta_2 \frac{M(t)\alpha'(t) - M'(t)\alpha(t)}{L(t)M^2(t)}\bm n_2(t).$$
	}
\end{remark}

Differentiating the above equation, we can get
$$\begin{aligned}
	\mathcal Ev_\gamma'(t) =& \alpha(t)\bm\mu(t) - \frac{\alpha'(t)M(t) - \alpha(t)M'(t)}{M^2(t)}\bm n_1(t) - \frac{\alpha(t)}{M(t)}\big(L(t)\bm n_2(t) + M(t)\bm\mu(t)\big)\\
	&- \delta_1\delta_2 \frac{\rm d}{{\rm d}t}\left(\frac{M(t)\alpha'(t) - M'(t)\alpha(t)} {L(t)M^2(t)}\right)\bm n_2(t) + \frac{\alpha'(t)M(t) - \alpha(t)M'(t)}{M^2(t)}\bm n_1(t)\\
	=& \left(-\frac{\alpha(t)L(t)}{M(t)} - \delta_1\delta_2 \frac{\rm d}{{\rm d}t}\left(\frac{M(t)\alpha'(t) - M'(t)\alpha(t)} {L(t)M^2(t)}\right)\right)\bm n_2(t).
\end{aligned}$$
So $(\mathcal Ev_\gamma, \bm n_1, \bm\mu)$ is a non-parabolic spatial hybrid framed curve.

Now we consider the inverse process of evolutes.
\begin{definition}
	{\rm  Let $(\gamma, \bm n_1, \bm n_2) : I \rightarrow \mathbb H^p \times \Delta$ be a non-parabolic spatial hybrid framed curve with the curvature $(L, M, 0, \alpha).$ The involute of $(\gamma, \bm n_1, \bm n_2)$ is defined by $\mathcal Inv_\gamma : I \rightarrow \mathbb H^p,$
		$$\mathcal Inv_\gamma(t) = \gamma(t) + f_1(t)\bm n_1(t) + f_2(t)\bm\mu(t),$$
		where $(f_1, f_2)$ is the solution of the system
		$$\left\{
		\begin{array}{l}
			f_1'(t) = \sigma\delta_1\delta_2 M(t)f_2(t),\\
			f_2'(t) = -M(t)f_1(t) - \alpha(t).
		\end{array}
		\right.$$
	}
\end{definition}

It obvious that
$$\begin{aligned}
	\mathcal Inv_\gamma'(t) =& \alpha(t)\bm\mu(t) + \sigma\delta_1\delta_2M(t)f_2(t)\bm n_1(t) + f_1(t)L(t)\bm n_2(t) + f_1(t)M(t)\bm\mu(t)\\
	&- \left(M(t)f_1(t) + \alpha(t)\right)\bm \mu(t) - \sigma\delta_1\delta_2 M(t)f_2(t)\bm n_1(t)\\
	=& f_1(t)L(t)\bm n_2(t).
\end{aligned}$$
So $(\mathcal Inv_\gamma, \bm n_1, \bm\mu)$ is a non-parabolic spatial hybrid framed curve. Note that $\bm n_1(t) \times \bm\mu(t) = -\sigma\bm n_2(t).$ Then $\{\bm n_1(t), \bm\mu(t), -\sigma\bm n_2(t)\}$ is a $g$-orthogonal frame along $\mathcal Inv_\gamma.$ The Frenet type formulas are
$$\begin{pmatrix}
	\bm n_1'(t)\\
	\bm\mu'(t)\\
	-\sigma\bm n_2'(t)
\end{pmatrix}
=
\begin{pmatrix}
	0 & M(t) & -\sigma L(t)\\
	-\sigma\delta_1\delta_2 M(t) & 0 & 0\\
	\sigma\delta_1\delta_2 L(t) & 0 & 0
\end{pmatrix}
\begin{pmatrix}
	\bm n_1(t)\\
	\bm\mu(t)\\
	-\sigma\bm n_2(t)
\end{pmatrix},$$
$$\mathcal Inv_\gamma'(t) = \big(-\sigma f_1(t)L(t)\big)\big(-\sigma \bm n_2(t)\big).$$

\begin{theorem}
	{\rm Let $(\gamma, \bm n_1, \bm n_2) : I \rightarrow \mathbb H^p \times \Delta$ be a non-parabolic spatial hybrid framed curve with the curvature $(L, M, 0, \alpha),$ where $L(t) \neq 0$ for any $t \in I.$ $\mathcal Inv_\gamma$ is its involute. Then the evolute of $(\mathcal Inv_\gamma, \bm n_1, \bm\mu)$ is $\gamma.$
	}
\end{theorem}
\begin{proof}
	$$\begin{aligned}
		&\mathcal Ev_{\mathcal Inv_\gamma}(t)\\
		=& \mathcal Inv_\gamma(t) - \frac{-\sigma f_1(t)L(t)} {-\sigma L(t)}\bm n_1(t) - \sigma\delta_1\delta_2\frac{\sigma L(t)\big(\sigma f_1(t)L(t)\big)' - \sigma L'(t) \sigma f_1(t)L(t)}{L^2(t)M(t)}\bm\mu(t)\\
		=& \mathcal Inv_\gamma(t) - f_1(t)\bm n_1(t) - \sigma\delta_1\delta_2\frac{f_1'(t)}{M(t)} \bm\mu(t)\\
		=& \gamma(t).
	\end{aligned}$$
\end{proof}

For a non-parabolic spatial hybrid framed curve, we can define its pedal and contrapedal curve by $g$-orthogonal projection.

\begin{definition}
	{\rm Let $(\gamma, \bm n_1, \bm n_2) : I \rightarrow \mathbb H^p \times \Delta$ be a non-parabolic spatial hybrid framed curve with the curvature $(L, M, 0, \alpha).$ $\bm p \in \mathbb H^p$ is a fixed point.
		\begin{itemize}
			\item [\rm (1)] The pedal curve of $(\gamma, \bm n_1, \bm n_2)$ relative to $\bm p$ is defined by $\mathcal Pe_{\gamma,p} : I \rightarrow \mathbb H^p,$
			$$\mathcal Pe_{\gamma,p}(t) = \bm p - \sigma\delta_1\delta_2 g(\bm p - \gamma(t), \bm n_2(t))\bm n_2(t).$$
			\item [\rm (2)] The contrapedal curve of $(\gamma, \bm n_1, \bm n_2)$ relative to $\bm p$ is defined by $\mathcal{CP}e_{\gamma,p} : I \rightarrow \mathbb H^p,$
			$$\mathcal{CP}e_{\gamma,p}(t) = \bm p - \delta_1\delta_2 g(\bm p - \gamma(t), \bm\mu(t))\bm\mu(t).$$
		\end{itemize}
	}
\end{definition}

There are relationships among evolutes, involutes, pedal and contrapedal curves.
\begin{theorem}
	{\rm Let $(\gamma, \bm n_1, \bm n_2) : I \rightarrow \mathbb H^p \times \Delta$ be a non-parabolic spatial hybrid framed curve with the curvature $(L, M, 0, \alpha),$ where $L(t) \neq 0$ for any $t \in I.$ $\bm p \in \mathbb H^p$ is a fixed point. Then the pedal curve of the evolute $(\mathcal Ev_\gamma, \bm n_1, \bm\mu)$ coincides with the contrapedal curve of $(\gamma, \bm n_1, \bm n_2).$ The contrapedal curve of the involute $(\mathcal Inv_\gamma, \bm n_1, \bm\mu)$ coincides with the pedal curve of $(\gamma, \bm n_1, \bm n_2).$	
	}
\end{theorem}
\begin{proof}
	$$\begin{aligned}
		\mathcal Pe_{\mathcal Ev_\gamma,p}(t) =& \bm p - \delta_1\delta_2 g(\bm p - \mathcal Ev_\gamma(t), \bm\mu(t))\bm\mu(t)\\
		=& \bm p - \delta_1\delta_2 g(\bm p - \gamma(t), \bm\mu(t))\bm\mu(t)\\
		=& \mathcal{CP}e_{\gamma,p}(t),\\
		\mathcal{CP}e_{\mathcal Inv_\gamma,p}(t) =& \bm p - \sigma\delta_1\delta_2 g(\bm p - \mathcal Inv_\gamma(t), \bm n_2(t))\bm n_2(t)\\
		=& \bm p - \sigma\delta_1\delta_2 g(\bm p - \gamma(t), \bm n_2(t))\bm n_2(t)\\
		=& \mathcal Pe_{\gamma,p}(t).
	\end{aligned}$$
\end{proof}

\section{Example}
\begin{example}
	{\rm Let $(\gamma, \bm n_1, \bm n_2) : [0, 2\pi] \rightarrow \mathbb H^p \times \Delta$ be
		$$\begin{aligned}
			\gamma(t) =& \left(\sin^3 t - \frac{\sqrt{10}}{4}\cos 2t \right)\bm i + (\sin^3 t)\bm\varepsilon + (\cos^3 t)\bm h,\\
			\bm n_1(t) =& (\cos t)\bm i + (\cos t)\bm \varepsilon + (\sin t)\bm h,\\
			\bm n_2(t) =& (3 + \sqrt{10}\sin t)\bm i + (\sqrt{10}\sin t)\bm \varepsilon - (\sqrt{10}\cos t)\bm h.
		\end{aligned}$$
		Then
		$$\bm\mu(t) = (\sqrt{10} + 3\sin t)\bm i + (3\sin t)\bm\varepsilon - (3\cos t)\bm h.$$
		The Frenet type formulas are
		$$\begin{pmatrix}
			\bm n_1'(t)\\
			\bm n_2'(t)\\
			\bm\mu'(t)
		\end{pmatrix}
		=
		\begin{pmatrix}
			0 & -\sqrt{10} & 3 \\
			\sqrt{10} & 0 & 0 \\
			3 & 0 & 0
		\end{pmatrix}
		\begin{pmatrix}
			\bm n_1(t)\\
			\bm n_2(t)\\
			\bm\mu(t)
		\end{pmatrix},
		~\gamma'(t) = \sin t \cos t \bm\mu(t).$$
		
		Then the evolute of $(\gamma, \bm n_1, \bm n_2)$ is
		$$\mathcal Ev_\gamma(t) = \left(\frac{2}{3}\sin^3 t - \frac{3\sqrt{10}}{20}\cos 2t\right)\bm i + \left(\frac{2}{3}\sin^3 t\right)\bm\varepsilon + \left(\frac{2}{3}\cos^3 t\right)\bm h.$$
		The involute of $(\gamma, \bm n_1, \bm n_2)$ is
		$$\mathcal Inv_\gamma(t) = \gamma(t) + f_1(t)\bm n_1(t) + f_2(t)\bm\mu(t),$$
		where $(f_1, f_2)$ satisfies
		$$\left\{
		\begin{array}{l}
			f_1'(t) = -3f_2(t),\\
			f_2'(t) = -3f_1(t) - \sin t \cos t.
		\end{array}
		\right.$$
		That is
		$$\left\{
		\begin{array}{l}
			f_1(t) = -\dfrac{3}{26}\sin 2t + c_1{\rm e}^{3t} + c_2{\rm e}^{-3t},\\
			f_2(t) = \dfrac{1}{13}\cos 2t - c_1{\rm e}^{3t} + c_2{\rm e}^{-3t},
		\end{array}
		\right.$$
		where $c_1, c_2 \in \mathbb R.$
		If we take $c_1 = c_2 = 0,$ the involute of $(\gamma, \bm n_1, \bm n_2)$ is
		$$\mathcal Inv_\gamma(t) = \left(\frac{10}{13}\sin^3 t - \frac{9\sqrt{10}}{52}\cos 2t\right)\bm i + \left(\frac{10}{13}\sin^3 t\right)\bm\varepsilon + \left(\frac{10}{13}\cos^3 t\right)\bm h.$$
		We draw them in Figure \ref{ex1a}.
		
		Take $\bm p = \bm 0 \in \mathbb H^p,$ the pedal curve of $(\gamma, \bm n_1, \bm n_2)$ relative to $\bm p$ is
		$$\mathcal Pe_{\gamma,p}(t) = \cos 2t \left(\left(-\frac{3\sqrt{10}}{4} - \frac{5}{2}\sin t\right)\bm i - \left(\frac{5}{2}\sin t\right)\bm\varepsilon + \left(\frac{5}{2}\cos t\right)\bm h\right).$$
		The contrapedal curve of $(\gamma, \bm n_1, \bm n_2)$ relative to $\bm p$ is
		$$\mathcal{CP}e_{\gamma,p}(t) = \frac{1}{2}\cos 2t \left((\sqrt{10} + 3\sin t)\bm i + (3\sin t)\bm\varepsilon - (3\cos t)\bm h\right).$$
		We draw them in Figure \ref{ex1b}.
		
		\begin{figure}[H]
			\centering
			\includegraphics[width=15cm]{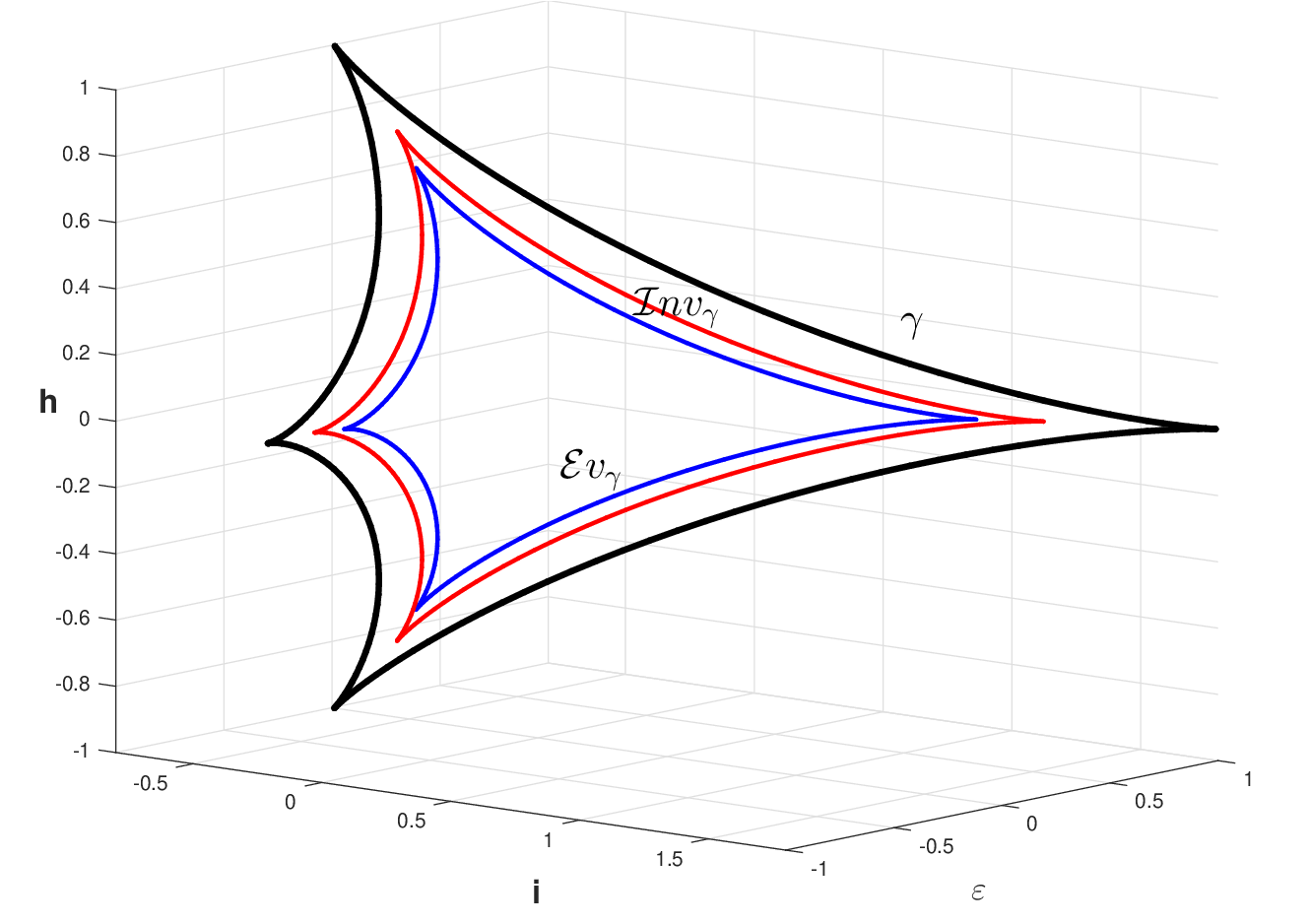}
			\caption{A non-parabolic spatial hybrid framed base curve (black) with its evolute (blue) and involute (red).}
			\label{ex1a}
		\end{figure}
		
		\begin{figure}[H]
			\centering
			\includegraphics[width=15cm]{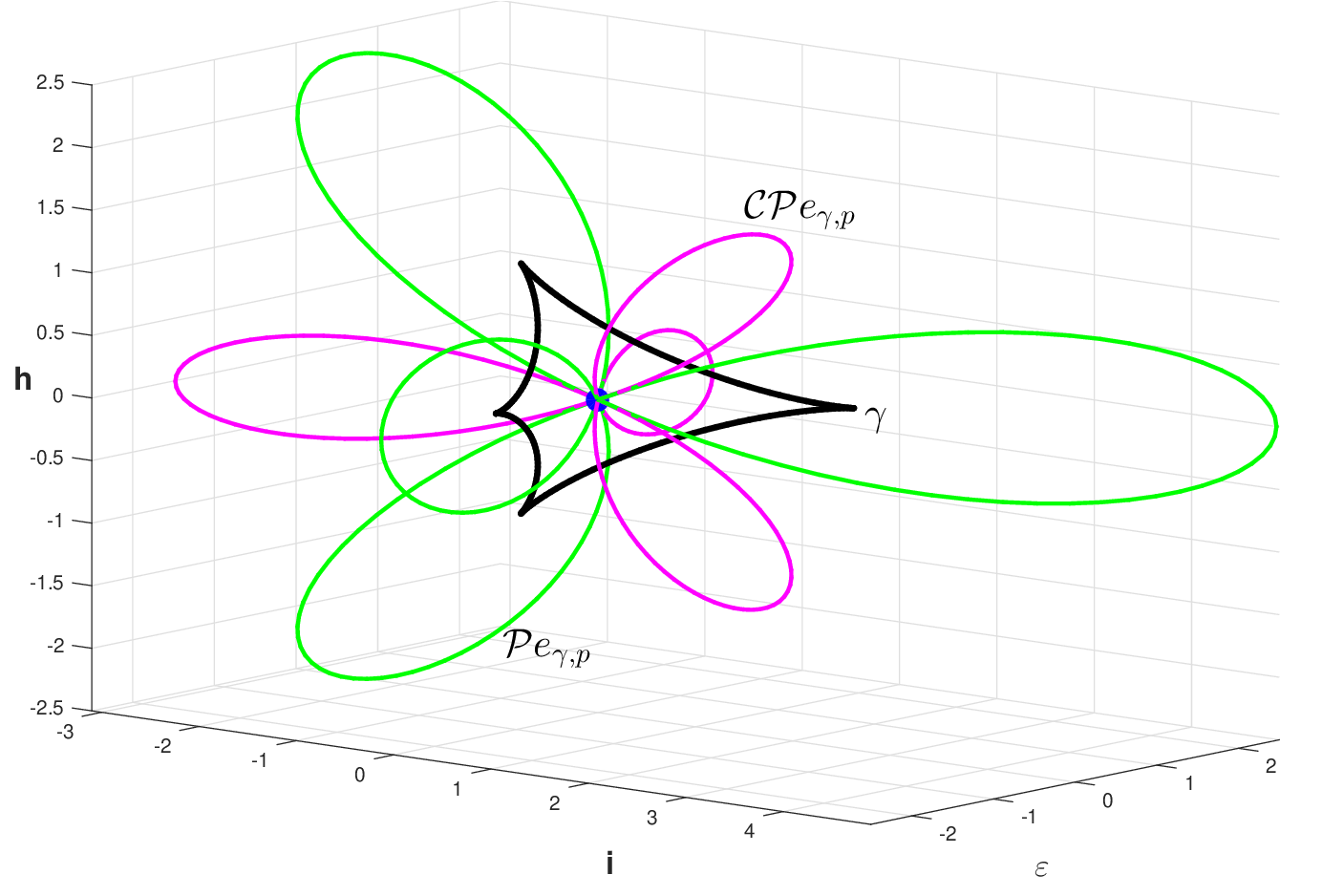}
			\caption{A non-parabolic spatial hybrid framed base curve (black) with its pedal curve (green) and contrapedal curve (magenta).}
			\label{ex1b}
		\end{figure}
	}
\end{example}

\subsection*{Acknowledgment}
This work was funded by Yanshan University Research Start-up Fund (Grant No. 8190857) and Science Research Project of Hebei Education Department (Grant No. QN2026104).


\end{document}